\documentclass[final,1p,times]{elsarticle}

\usepackage{amssymb}
 \usepackage{amsthm}
\usepackage{amscd}
\usepackage{amsmath}
\usepackage{amsfonts}
\usepackage{amssymb}
\usepackage{graphicx}
\newtheorem{theorem}{Theorem}

\newtheorem{lemma}[theorem]{Lemma}

\usepackage{mathrsfs}
\usepackage{titletoc}

\newcommand{\bs}{\backslash}
\renewcommand{\d}{\delta}

\newcommand{\D}{\Delta}
\newcommand{\ra}{\rightarrow}

\newcommand{\f}{\frac}
\newcommand{\al}{\alpha}
\newcommand{\g}{\gamma}

\renewcommand{\l}{\lambda}

\newcommand{\be}{\begin{equation}}
\renewcommand{\ra}{\rightarrow}
\newcommand{\ee}{\end{equation}}
\newcommand{\bea}{\begin{eqnarray}}
\newcommand{\eea}{\end{eqnarray}}
\newcommand{\bna}{\begin{eqnarray*}}
\newcommand{\ena}{\end{eqnarray*}}

\renewcommand{\O}{\Omega}
\renewcommand{\le}{\left}
\newcommand{\ri}{\right}

\newcommand{\ve}{\epsilon}
\newcommand{\msh}{\mathscr{H}}
\newcommand{\mss}{\mathscr{S}}
\newcommand{\na}{\nabla}
\newcommand{\B}{\mathbb{B}}
\newcommand{\ba}{\beta}
\newcommand{\ib}{\int_{\mathbb{B}}}
\newcommand{\iy}{\infty}

\journal{***}

\begin{document}

\begin{frontmatter}

\title{Extremal functions for a singular Hardy-Moser-Trudinger inequality}

\author{Songbo Hou}
\ead{housb10@163.com}
\address{Department of Applied Mathematics, College of Science, China Agricultural University,  Beijing, 100083, P.R. China}

\begin{abstract}
In this paper, using blow-up analysis, we prove a singular Hardy-Morser-Trudinger inequality, and find its
 extremal functions. Our results extend  those of Wang-Ye (Adv. Math. 2012), Yang-Zhu ( Ann. Glob. Anal. Geom. 2016), Csat\'{o}- Roy  (Calc. Var. 2015),  and Yang-Zhu (J. Funct. Anal. 2017).
\end{abstract}


\end{frontmatter}

\section{Introduction and main results}
Let $\O\subset \mathbb{R}^2$ be a smooth bounded domain, and $W^{1, 2}_0(\O)$ be the usual Sobolev  space.  The classical Moser-Trudinger inequality  \cite{Mos71, Tru67} states that
\be\label{mti}
\sup_{u\in W^{1,2}_0 (\O),\,||\na u||_2\leq 1}\int_{\O}e^{\g u^2}dx <\iy,
\ee
 for any $\g\leq 4\pi$. If $\g>4\pi$, then the supremum is infinite although the integrals in (\ref{mti}) are finite. In this sense, the inequality (\ref{mti}) is sharp. It plays an important role in geometric analysis and partial differential equations. Let $\l_1({\O})$ be the first eigenvalue of the Laplacian operator with respect to Dirichlet boundary condition. Adimurthi and O. Druet \cite{AD04} proved that for any $\al$, $0\leq \al <\l_1(\O)$,
 \be\label{adi}
 \sup_{u\in W^{1,2}_0 (\O),\,||\na u||_2\leq 1}\int_{\O}e^{4\pi u^2(1+\al ||u||^2_2)}dx <\iy,
 \ee
whereas for any $\al\geq \l_1(\O)$, the above supremum is infinity. The analogs of (\ref{adi}) were obtained on a compact Riemannian surface \cite{Yang06} and
on a high dimeniosnal Euclidean domain \cite{Yang07}. Clearly, the inequality (\ref{adi}) is stronger than the inequality (\ref{mti}).
C. Tintarev \cite{Tin14} obtained an improvement of Moser-Trudinger inequality as the following:
\be\label{tine}
\sup_{\int_{\O}|\na u|^2dx-\int_{\O}V(x)u^2dx\leq 1, u\in C_0^{\infty}(\O)}\int_{\O}e^{4\pi u^2}dx<\iy,
\ee
where $V(x)>0$ is a specific class of potentials. Let us write for any $0\leq \al<\l_1(\O)$,
\be\label{und}
||u||_{1,\al}=\le(\int_{\O}|\na u|^2dx-\al\int_{\O}u^2dx\ri)^{1/2}.\ee
Then a special case of (\ref{tine}) is the following:
\be\label{tsc}
\sup_{u\in W^{1,2}_0(\O),\,||u||_{1,\al}\leq 1}\int_{\O}e^{4\pi u^2}dx<\iy,
\ee
 where $0\leq \al<\l_1(\O)$. Note that the inequality (\ref{tsc}) is stronger than (\ref{adi}). It was shown by Y. Yang \cite{Yang15}
 that the supremum in (\ref{tsc})
can be attained by some function $u_0\in W^{1,2}_0(\O)\cap C^1(\overline{\O})$
 with $||u_0||_{1,\al}=1$.\\

For singular Moser-Trudinger inequalities, Adimurthi and K. Sandeep \cite{A-S} proved that for  any $\ba$, $0\leq \ba <1$, there holds
\be\label{asi}
\sup_{u\in W^{1,2}_0(\O),\,||\na u||_{2}\leq 1}\int_{\O}\f{e^{4\pi (1-\ba )u^2}}{|x|^{2\ba}}dx<\iy.
\ee
An analog of (\ref{asi}) was established by Adimurthi and Y. Yang \cite {AY10} in the entire Euclidean space. The existence of  the extremal function of (\ref{asi}) was proved by G. Csat\'{o} and  P. Roy  \cite{CSP15}. Let $\B\subset\mathbb{R}^2$ be the standard unit disc. Yuan and  Zhu
\cite{YuZ} proposed the following  Adimurthi-Druet type inequality:
\be\label{yzi}
\sup_{u\in W^{1,2}_0(\B),\,||\na u||_2\leq 1}\ib\f{e^{4\pi(1-\ba)u^2(1+\al ||u||^2_{2,\ba})}}{|x|^{2\ba}}dx<\iy,
\ee
 where $0\leq \ba <1$, $||u||_{2,\ba}=\le(\ib|x|^{-2\ba}u^2dx\ri)^{1/2}$, $0\leq \al <\l_{1, \ba}(\B)$ with
 $$\l_{1,\ba}(\B)=\inf_{u\in W^{1,2}_0(\B), ||u||_{2,\ba}=1}\ib |\na u|^2dx.$$
 Recently,  (\ref{asi}) was generalized by Y. Yang and X. Zhu \cite{YZ16} to the following: there holds for any
 $\alpha$, $0\leq\alpha<\lambda_1(\Omega)$,
 \be\label{yzi}
 \sup_{u\in W^{1,2}_0(\O),\,||u||_{1,\al}\leq 1}\int_{\O}\f{e^{4\pi (1-\ba )u^2}}{|x|^{2\ba}}dx<\iy,
 \ee
where $||u||_{1,\al}$ is defined as in (\ref{und}).  Moreover,  the above supremum can be attained. A related result can be found in \cite{LY16}.

Another important inequality in analysis is the Hardy inequality, which says that
\be\label{hari}
\ib|\na u|^2dx-\ib\f{u^2}{(1-|x|^2)^2}dx\geq 0,\quad\forall u\in W^{1,2}_0(\B), \ee
where $\B$ is the unit disc in $\mathbb{R}^2$.
It was proved by H. Brezis and M. Marcus \cite{HM97} that there exists a constant $C>0$  such that
  \be\label{bmmi}
  \ib|\na u|^2dx-\ib\f{u^2}{(1-|x|^2)^2}dx\geq C\ib u^2dx,\quad\forall u\in W^{1,2}_0(\B).
  \ee
Hence
$$||u||_{\msh}=\left(\int_{\mathbb{B}}|\na u|^2dx-\int_{\mathbb{B}}\f{u^2}{(1-|x|^2)^2}dx\right)^{1/2}$$
define a norm over $W^{1,2}_0(\B)$. Denote by $\msh$  the completion of $C_0^{\iy}(\B)$ under the norm $||\cdot||_\msh$. Then $\msh$ is a Hilbert space.
It was observed by Manchi-Sandeep \cite {MS08} and Wang-Ye \cite{GY12} that
$$W^{1,2}_0(\B)\subset \msh \subset \cap_{p\geq 1}L^p(\B).$$
  A Hardy-Moser-Trudinger inequality was first established by  G. Wang and D. Ye \cite{GY12}:
\be \label{wyi}
\sup_{u\in \msh,\, ||u||_\msh\leq 1} \ib e^{4\pi u^2}dx<\iy.
\ee
Moreover,  the above supremum can be attained.

We slightly abuse some notations and write
$$\l_1(\B)=\inf_{u\in\msh,u\not\equiv 0}\f{||u||^2_\msh}{||u||^2_2},$$
and $$||u||_{\msh,\al}=\le(||u||_\msh-\al ||u||^2_2\ri)^{1/2},$$ where $0\leq \al <\l_1(\mathbb{B})$.
In \cite{YZ15}, Yang and  Zhu improved the result of G. Wang and D. Ye as below:

\be\label{yzi}
\sup_{u\in \msh,\, ||u||_{\msh, \al}\leq 1} \ib e^{4\pi u^2}dx<\iy.\ee
Moreover, the extremal function for the above supremum exists.\\

 Our aim is to extend (\ref{yzi}) to a singular version of the Hardy-Moser-Trudinger inequality. Now, 
 the main result of this paper
 can be stated as follows:
\begin{theorem} \label{mai1} Let $0\leq\ba<1$ be fixed.  Then for  any $\al$, $0\leq\al<\l_1(\B)$, there holds
\be\label{hyi}\sup_{u\in \msh,\, ||u||_{\msh, \al}\leq 1} \ib \f{e^{4\pi (1-\ba )u^2}}{|x|^{2\ba}}dx<\iy.\ee
Furthermore the supremum can be achieved by some function $u_0\in \msh$ with $||u_0||_{\msh,\al}=1$.
\end{theorem}

 When $\ba=0$, the inequality (\ref{hyi}) is reduced to (\ref{yzi}). The existence of extremal functions for Moser-Trudinger inequality originated in \cite{CC86}. This result was generalized by  M. Struwe \cite{str88}, F. Flucher \cite{flu92}, K. Lin \cite{Lin96}, W. Ding, J. Jost, J. Li and G. Wang \cite{DJL97}, Adimurthi and M. Struwe \cite{AST00}, Y. Li \cite{Li01}, Adimurthi and O. Druet \cite{AD04},  and so on. Compared with \cite{YZ15}, there are difficulties caused by the term $|x|^{-2\ba}$ in the process of blow-up analysis. Here we employ a  classification theorem of W. Chen and C. Li \cite{CL95} which was also used in \cite{YZ16}, while  another
  classification result \cite{CL91} was also used in \cite{YZ15}. We derive an upper bound of the Hardy-Moser-Trudinger functionals by Onofri's inequality (\cite {IM16}, Theorem 1.1 ), while an upper bound was obtained via the capacity estimate in \cite{YZ15}. The proof of Theorem 1 is composed of three steps. The first step is to reduce the problem on radially non-increasing functions and derive the associated Euler-Lagrange equation. The second step is the blow-up analysis. If the blow-up occurs, we analyze the asymptotic behavior of maximizing sequences near and away from  the blow-up point. Then we estimate an upper bound of subcritical functionals. The final spep is to construct test functions and get a contradiction with the upper bound derived in the previous step, which implies that blow-up can not occur.  This completes the proof of Theorem \ref{mai1}.
\section{Proof of main results}\label{Sec 2}

\subsection{The singular subcritical functionals}
In this subsection, we will  prove the existence of  the maximizers of subcritical functions. We recall  Wang-Ye's result for our use later:

\begin{lemma} \label{GYL}(Wang-Ye \cite{GY12}) Let
$$\mss_0=\le\{u\in C_0^{\iy}(\B):u(x)=u(r)\,\,{\rm with}\,\,r=|x|,u'\leq 0\ri\}$$ and $\mss$ be the closure of $\mss_0$ in $\msh$. Then $\mss$ is embedded  continuously in $W^{1,2}_{\rm loc}(\B)\cap C^{0,\f{1}{2}}_{\rm loc}(\B\bs \{0\})$. Moreover, for any $p\geq 1$, $\mss$ is embedded compactly in $L^p(\B)$.
\end{lemma}
Then we  perform  variation in $\mss$ instead of $\msh$ and get the following:
\begin{lemma}\label{Lemma 1} Assume $0\leq \alpha\leq \l_1({\mathbb{B}})$. Then for any $\ve$, $0<\ve<1-\beta$,
there exists some  $u_{\ve}\in \mathscr{S}\cap C^{\infty}(\mathbb{B}\backslash\{0\})\cap C^{0}(\overline{\mathbb{B}})$ such that
$||u_{\ve}||_{\msh,\alpha}=1$ and
\be\label{le1}
\int_{\mathbb{B}}\f{e^{4\pi(1-\beta-\ve)u^{2}_{\ve}}}{|x|^{2\beta}}dx=\sup_{u\in\msh,\,||u||_{\msh,\alpha}\leq 1}\int_{\mathbb{B}}\f{e^{4\pi(1-\beta-\ve)u^{2}}}{|x|^{2\beta}}dx.
\ee

\end{lemma}
\begin{proof}
First, we  prove that
\be\label{sup}
\sup_{u\in\msh,\,||u||_{\msh,\alpha}\leq 1}\int_{\mathbb{B}}\f{e^{4\pi(1-\beta-\ve)u^{2}}}{|x|^{2\beta}}dx=
\sup_{u\in\mathscr{S},\,||u||_{\msh,\alpha}\leq 1}\int_{\mathbb{B}}\f{e^{4\pi(1-\beta-\ve)u^{2}}}{|x|^{2\beta}}dx,
\ee
which reduces our problem on radially symmetric functions.

For any $u\in C_0^{\infty}({\mathbb{B}})$, denote by $u^{*}$ the radially nonincreasing rearrangement of $u$ with respect to the
standard  hyperbolic metric $d\mu=\f{1}{(1-|x|^2)^2}dx$. The argument in \cite{Ba94} leads to
$$\int_{\mathbb{B}}|\na u^*|^2dx\leq \int_{\mathbb{B}}|\na u|^2dx,$$
$$\ib {u^{*}}^2dx=\ib u^2dx$$ and
$$\int_{\mathbb{B}}\f{{u^{*}}^2}{(1-|x|^2)^2}dx=\int_{\mathbb{B}}\f{u^{2}}{(1-|x|^2)^2}dx.$$
Thus, $||u||_{\msh,\alpha}\leq 1$ implies $||u^*||_{\msh,\alpha}\leq 1$. Using the  Hardy-Littlehood inequality and noticing that the rearrangement of $\f{(1-|x|^2)^2}{|x|^{2\beta}}$ is just itself, we get
\bna
\int_{\mathbb{B}}\f{e^{4\pi(1-\beta-\ve)u^{2}}}{|x|^{2\beta}}dx &=& \int_{\mathbb{B}}e^{4\pi(1-\beta-\ve)u^{2}}\f{(1-|x|^2)^2}{|x|^{2\beta}}d\mu\\
&\leq & \int_{\mathbb{B}}e^{4\pi(1-\beta-\ve){u^{*}}^2}\left(\f{(1-|x|^2)^2}{|x|^{2\beta}}\right)^* d\mu\\
&=& \int_{\mathbb{B}}e^{4\pi(1-\beta-\ve){u^{*}}^2}\f{(1-|x|^2)^2}{|x|^{2\beta}} d\mu\\
&=&\int_{\mathbb{B}}\f{e^{4\pi(1-\beta-\ve){u^{*}}^2}}{|x|^{2\beta}}dx.
\ena
Thus,
$$\sup_{u\in C_0^{\infty}({\mathbb{B}}),\,||u||_{\msh,\alpha}\leq 1}\int_{\mathbb{B}}\f{e^{4\pi(1-\beta-\ve)u^{2}}}{|x|^{2\beta}}dx \leq
\sup_{u\in\mathscr{S},\,||u||_{\msh,\alpha}\leq 1}\int_{\mathbb{B}}\f{e^{4\pi(1-\beta-\ve)u^{2}}}{|x|^{2\beta}}dx.$$
Combining the density of $C_0^{\infty}({\mathbb{B}})$ in $\msh$, we see that (\ref{sup}) holds.

Next,  we prove that for any $\ba$, $0\leq \ba<1$ and any $\ve$, $0<\ve<1-\ba$, there holds
\be\label{sub}
\sup_{u\in\mathscr{S},\,||u||_{\msh}\leq 1}\int_{\mathbb{B}}\f{e^{4\pi(1-\ba-\ve)u^{2}}}{|x|^{2\ba}}dx<+\infty.
\ee
We modify the proof of Theorem 3 in \cite{GY12}. Let $u\in \mss_0$. Define $$A_{u}(r)=\f{1}{\pi r^2}\int_{\mathbb{B}_r}\f{u^2}{(1-|x|^2)^2}dx.$$
We may assume $u(0)>1$, otherwise
$\int_{\mathbb{B}}\f{e^{4\pi(1-\ba-\ve)u^{2}}}{|x|^{2\ba}}dx\leq \f{\pi e^{4\pi(1-\ba-\ve)}}{1-\ba}$.
Define
$r_1=\inf\{r>0\,|\,u(r)\leq 1\}>0$.
By Lemma 3 and Lemma 4 in \cite{GY12}, we have for any $r\leq r_1$,
$$\int_{\mathbb{B}_r}|\nabla u|^2dx\leq 1-||u||_{\msh(\mathbb{B}_r^c)}^2+C\pi r^2 u(r)^2\leq 1-||u||_{\msh(\mathbb{B}_r^c)}^2+\f{C}{2}r||u||_{\msh(\mathbb{B}_r^c)}^2,$$
where $\mathbb{B}_r^c=\mathbb{B}\setminus \mathbb{B}_r$,   $||\cdot||_{\msh(\B_r^c)}$ is  the norm $||\cdot||_{\msh}$ on $\mathbb{B}_r^c$ and $C$ is a positive constant independent of $u$.
Hence for $r_2\in(0,r_1]$ small enough, independent of $u$, there holds
$$||\nabla u||_{L^2(\mathbb{B}_{r_2})}
\leq 1.$$ Moreover, $u(r_2)$ has an upper bound independent of $u$.
By the singular Moser-Trudinger inequality (\ref{asi}), we have
$$\int_{\mathbb{B}_{r_2}}\f{e^{4\pi(1-\ba)\left[u(r)-u(r_{2})\right]^2}}{|x|^{2\ba}}dx=\int_{\mathbb{B}}\f{e^{4\pi(1-\ba){\left[u(r)-u(r_{2})\right]^+}^2}}{|x|^{2\ba}}dx<+\infty,$$
where $\left[u(r)-u(r_{2})\right]^+=\max\{u(r)-u(r_{2}),0\}$.
Then we have for any $r\leq r_2$,
$$4\pi(1-\ba-\ve)u(r)^2\leq 4\pi(1-\ba)\left[u(r)-u(r_2)\right]^2+8\pi(1-\ba) u(r)u(r_2)-4\pi\ve u^2(r)\leq 4\pi(1-\ba)\left[u(r)-u(r_2)\right]^2+C_{\ve},$$
where $C_{\ve}$ is a positive constant depending only on $\ve$. Hence we get
\bna\int_{\B}\f{e^{4\pi(1-\ba-\ve)u^2}}{|x|^{2\ba}}dx&=&\int_{\B_{r_2}}\f{e^{4\pi(1-\ba-\ve)u^2}}{|x|^{2\ba}}dx+\int_{\B_{r_2}^c}\f{e^{4\pi(1-\ba-\ve)u^2}}{|x|^{2\ba}}dx\\
&\leq & \int_{\B_{r_2}}\f{e^{4\pi(1-\ba)\left[u-u(r_2)\right]^2+C_{\ve}}}{|x|^{2\ba}}dx+\f{\pi(1-r_2^{2-2\ba})}{1-\ba}e^{4\pi(1-\ba-\ve)u(r_2)^2}\\
&< & +\infty.
\ena
One can see that (\ref{sub}) holds true.

We use a method of variation to prove (\ref{le1}). Choose a maximizing sequence $u_j\in \mss$ with $||u_j||_{\msh,\al}\leq 1$ such that
\be
\int_{\B}|x|^{-2\beta}e^{4\pi(1-\beta-\ve)u_j^2}dx\rightarrow  \sup_{u\in\msh,\,||u||_{\msh,\alpha}\leq 1}\int_{\mathbb{B}}|x|^{-2\beta}e^{4\pi(1-\beta-\ve)u^{2}}dx
\ee
as $j\ra \infty$.
Noting $0\leq \al<\l_1(\B)$, we obtain
$$\le(1-\f{\al}{\l_1(\B)}\ri)||u_j||_{\msh}^2\leq 1.$$
This implies that $u_j$ is bounded in $\msh$. There exists some $u_{\ve}\in \mss$ such that up to a subsequence,
\bna
u_j\rightharpoonup u_{\ve} &\mathrm{weakly \,\,in\,\,\,} &\msh,\\
u_j\ra u_{\ve}  & \mathrm{strongly\,\,in} &L^p(\B),\,\,\forall p\geq 1,\\
u_j\ra u_{\ve}  & \mathrm{a.e. \,\,in\,\,\,\,\;\;\;\,}&\B.
\ena
Clearly, $||u_{\ve}||_{\msh,\al}\leq 1$ since $||u_{j}||_{\msh,\al}\leq 1$.

Because  $\langle u_j,u_{\ve}\rangle_{\msh}\ra ||u_{\ve}||^2_{\msh}$ and $||u_j||_2\ra ||u_{\ve}||_2$, we have
\bea
||u_j-u_{\ve}||_{\msh}^2&=&\langle u_j-u_{\ve}, u_j-u_{\ve}\rangle_{\msh}\nonumber\\
&=&||u_j||_{\msh}^2+||u_{\ve}||_{\msh}^2-2\langle u_j, u_{\ve}\rangle_{\msh}\nonumber\\
&=&||u_j||_{\msh}^2-||u_{\ve}||_{\msh}^2+o_j(1)\nonumber\\
&\leq & 1-||u_{\ve}||_{\msh,\al}^2+o_j(1).
\eea
By the H\"{o}lder inequality, we have for  $1<p<1/\beta$,
\bea\label{eet}
\int_{\B}|x|^{-2\ba p}e^{4\pi(1-\ba-\ve)pu_j^2}dx&\leq&\ib|x|^{-2\ba p}e^{4\pi(1-\ba-\ve)p(1+\d)(u_j-u_{\ve})^2+4\pi(1-\ba-\ve)p(1+1/\d)u_{\ve}^2}dx\nonumber\\
&\leq &\le(  \ib |x|^{-2\ba p}e^{4\pi(1-\ba-\ve)pq(1+\d)(u_j-u_{\ve})^2}dx\ri)^{1/q}\nonumber\\
&&\times \le(  \ib |x|^{-2\ba p}e^{4\pi(1-\ba-\ve)pq'(1+1/\d)u_{\ve}^2}dx\ri)^{1/q'}\nonumber\\
&\leq & \le(  \ib |x|^{-2\ba p}e^{4\pi(1-\ba-\ve)pq(1+\d)(u_j-u_{\ve})^2}dx\ri)^{1/q}\nonumber\\
&&\times \le(\ib |x|^{-2\ba p s}dx\ri)^{1/(q's)}\nonumber\\
&&\times\le(\ib e^{4\pi(1-\ba-\ve)pq's'(1+1/\d)u_{\ve}^2}dx\ri)^{1/(q's')},
\eea
where $\d>0$, $1/q+1/q'=1$, $1/s+1/s'=1$.

We choose $p$, $q$, $1+\d$, $s$ sufficiently close to $1$, and $\d_1<\ve$ such that

$$4\pi(1-\ba-\ve)pq(1+\d)(u_j-u_\ve)^2<4\pi(1-\ba p-\d_1)\f{(u_j-u_\ve)^2}{||u_j-u_\ve||^2_\msh},\ \,\,\ba ps<1.$$
Lemma 4 in \cite{YZ15} indicates  that
for any $\g>0$ and any $u\in{\mss}$,  there holds
\be\label{zhu}
\ib e^{\g u^2}dx<+\infty.
\ee
Combining (\ref{sub}), (\ref{eet}) and (\ref{zhu}), we conclude that $|x|^{-2\ba}e^{4\pi(1-\ba-\ve)u_j^2}$ is bounded in $L^p(\B)$ for some $p>1$, which together with $|x|^{-2\ba}e^{4\pi(1-\ba-\ve)u_j^2}\ra |x|^{-2\ba}e^{4\pi(1-\ba-\ve)u_{\ve}^2}$ a.e. as $j\ra \infty$ implies that
\be
\lim_{j\ra \iy}\ib |x|^{-2\ba}e^{4\pi(1-\ba-\ve)u_j^2}dx=\ib |x|^{-2\ba}e^{4\pi(1-\ba-\ve)u_{\ve}^2}dx.
\ee
We claim that $||u_\ve||_{\msh,\al}=1$. Otherwise, we have $||u_\ve||_{\msh,\al}<1$. Hence
\bna\sup_{u\in\msh,\,||u||_{\msh,\alpha}\leq 1}\int_{\mathbb{B}}\f{e^{4\pi(1-\beta-\ve)u^{2}}}{|x|^{2\beta}}dx&=&\ib |x|^{-2\ba}e^{4\pi(1-\ba-\ve)u_{\ve}^2}dx\\
&<&\ib |x|^{-2\ba}e^{4\pi(1-\ba-\ve)u_{\ve}^2/||u_\ve||_{\msh,\al}^2}dx\\
&\leq &\sup_{u\in\msh,\,||u||_{\msh,\alpha}\leq 1}\int_{\mathbb{B}}\f{e^{4\pi(1-\beta-\ve)u^{2}}}{|x|^{2\beta}}dx,\ena
which is a contradiction.

A straightforward calculation shows that $u_{\ve}$ satisfies the following Euler-Lagrange equation
\be\label{Euler-Lagrange}\le\{\begin{array}{lll}-\Delta u_\epsilon-\f{u_\ve}{(1-|x|^2)^2}- \al u_\epsilon=\f{1}{\lambda_\epsilon}|x|^{-2\beta}u_\epsilon
e^{4\pi(1-\beta-\epsilon) u_\epsilon^2}\quad{\rm in}\quad \B,\\[1.2ex]
u_\epsilon\geq 0\quad {\rm in}\quad\mathbb{B},\\[1.2ex]
\lambda_\epsilon=\int_{\B} |x|^{-2\beta}u_\epsilon^2e^{4\pi(1-\beta-\epsilon) u_\epsilon^2}dx.
\end{array}\ri.\ee

Applying standard elliptic estimates to (\ref{Euler-Lagrange}),  we get $u_\ve\in C^{\iy}(\B\bs \{0\})$. Observing $u_\ve\in \mss$ , we have $u_\ve \in C^0(\overline{\B})$.
\end{proof}
\subsection{Blow-up analysis}
We use the method of blow-up analysis to describe the asymptotic behavior of the maximizers $u_\ve$ in Lemma \ref {Lemma 1}.
Note that $||u_\ve||_{\msh,\al}=1$, so $u_\ve$ is bounded in $\msh$. Thus, there exists $u_0\in\mss$ such that up a subsequence,

\bna
u_\ve\rightharpoonup u_{0} &&{\rm weakly}\quad{\rm in}\quad\msh,\\
u_\ve\ra u_{0}  && {\rm strongly}\quad{\rm in } \quad L^p(\B),\quad\forall p\geq 1,\\
u_\ve\ra u_{0}  && {\rm a.e.}\quad{\rm in}\quad\B.
\ena
It is clear that
$$||u_0||_{\msh,\al}\leq \liminf_{\ve\ra 0}||u_\ve||_{\msh,\al}=1.$$

Denote $c_\ve=u_\ve(0)=\max_{\B}u_\ve$. If $c_\ve$ is bounded, then the Lebesgue  dominated convergence theorem yields that
\bna
\ib|x|^{-2\ba}e^{4\pi(1-\ba-\ve)u_0^2}dx=\lim_{\ve\ra 0}\ib |x|^{-2\ba}e^{4\pi(1-\ba-\ve)u_\ve^2}dx=\sup_{u\in\msh,\,||u||_{\msh,\alpha}\leq 1}\int_{\mathbb{B}}\f{e^{4\pi(1-\beta)u^{2}}}{|x|^{2\beta}}dx.
\ena
Thus,  $u_0$
 is the desired maximizer.  Now, we assume that
 \be\label{cwq}
 c_\ve\ra +\infty \quad{\rm as}\quad  \ve\ra 0.
 \ee
The simple inequality $e^t\leq 1+te^t$ implies
 $$\ib|x|^{-2\ba}e^{4\pi(1-\ba-\ve)u_\ve^2}dx\leq \ib |x|^{-2\ba}dx+4\pi\l_\ve.$$
 This together with (\ref{le1}) yields
 \be\label{lbd}
 \liminf_{\ve\ra 0}\l_\ve>0.
 \ee
 Now, we claim that $u_0\equiv 0$. Otherwise, $||u_0||_{\msh,\al}>0$.  Calculate
 \bea
||u_\ve-u_{0}||_{\msh}^2&=&\langle u_\ve-u_{0}, u_\ve-u_{0}\rangle_{\msh}\nonumber\\
&=&||u_\ve||_{\msh}^2+||u_{0}||_{\msh}^2-2\langle u_\ve, u_{0}\rangle_{\msh}\nonumber\\
&=&||u_\ve||_{\msh,\al}^2-||u_{0}||_{\msh,\al}^2+o_\ve(1)\nonumber\\
&\leq & 1-||u_{0}||_{\msh,\al}^2+o_\ve(1).
\eea
 By  the similar estimates as in (\ref{eet}), we get $|x|^{-2\ba}e^{4\pi(1-\ba-\ve)u_\ve^2}$ is bounded in $L^p(\B)$ for some $p>1$.
  In view of (\ref{lbd}), applying elliptic estimates to (\ref{Euler-Lagrange}), we have  that $u_\ve$ is bounded in $C^0_{\rm loc}(\B)$. This contradicts (\ref{cwq}).

  Set \be
  r_\ve=\sqrt{\l_\ve}c_\ve^{-1}e^{-2\pi(1-\ba-\ve)c_\ve^2}.
  \ee
Note that $u_\ve\ra  0$ in $L^q(\B)$ for any $q\geq 1$.  Then,  by the H\"{o}lder inequality and (\ref{sub}), we have
\bea\label{rve}
\l_\ve=\ib|x|^{-2\ba}u_\ve^2e^{4\pi(1-\ba-\ve )u_\ve^2}dx\leq e^{4\pi\d c_\ve^2}\ib|x|^{-2\ba}u_\ve^2e^{4\pi(1-\ba-\ve-\d )u_\ve^2}dx \leq Ce^{4\pi\d c_\ve^2}
\eea
for $0<\d <1-\ba$, where the constant $C$ is independent of $u_\ve$. This leads to
\be\label{rgz}
r_\ve^2\leq Cc_\ve^{-2}e^{-4\pi(1-\ba-\ve-\d)c_\ve^2}\ra 0,\quad {\rm as}\quad \ve\ra 0.\ee
Let $\O_\ve=\{x\in \mathbb{R}^2:r_\ve^{1/(1-\ba)}x\in \B\}$.  We define  two blow-up sequences:
\be\label{bus}
\psi_\ve(x)=c_\ve^{-1}u_\ve(r_\ve^{1/(1-\ba)}x),\quad \varphi_\ve(x)=c_\ve(u_\ve(r_\ve^{1/(1-\ba)}x)-c_\ve).\ee
Then $\psi_\ve$ satisfies the following equation:
\be\label{pse}
-\D\psi_\ve= \f{\psi_\ve r_\ve^{2/(1-\ba)}}{(1-r_\ve^{2/(1-\ba)}|x|^2)^2}+\al r_\ve^{2/(1-\ba)}\psi_\ve
+c_\ve^{-2}|x|^{-2\ba}\psi_\ve e^{4\pi(1-\ba-\ve)(1+\psi_\ve)\varphi_\ve}\quad{\rm in }\quad \O_\ve.
\ee
By (\ref{rgz}), we have $r_\ve\ra 0$, hence $\O_\ve\ra \mathbb{R}^2$.   Using elliptic estimates to (\ref{pse}), we conclude that $\psi_\ve\ra\psi$ in
$C^1_{\rm loc}(\mathbb{R}^2\bs\{0\})\cap C^0_{\rm loc}(\mathbb{R}^2)$, where $\psi$ is a distributional harmonic function.
Clearly,  $\psi(0)=\lim_{\ve\ra 0}\psi_\ve(0)=1$. The Liouville theorem implies that $\psi\equiv 1$ on $\mathbb{R}^2$.  Hence we have
\be\label{pgo}
\psi_\ve\ra 1 \quad {\rm in}\quad C^1_{\rm loc}(\mathbb{R}^2\bs\{0\})\cap C^0_{\rm loc}(\mathbb{R}^2).\ee

A straightforward computation shows that
\be\label{phe}
-\D\varphi_\ve= \f{\psi_\ve c_\ve^2r_\ve^{2/(1-\ba)}}{(1-r_\ve^{2/(1-\ba)}|x|^2)^2}+\al c_\ve^2 r_\ve^{2/(1-\ba)}\psi_\ve
+|x|^{-2\ba}\psi_\ve e^{4\pi(1-\ba-\ve)(1+\psi_\ve)\varphi_\ve}\quad{\rm in }\quad \O_\ve.
\ee
In view of  (\ref{pgo}) and $\varphi_\ve(x)\leq \varphi_\ve(0)=0$ for all $x\in\O_\ve$, we have by the elliptic estimates that
\be\label{vgh}
\varphi_\ve\ra \varphi_0 \quad {\rm in}\quad C^1_{\rm loc}(\mathbb{R}^2\bs\{0\})\cap C^0_{\rm loc}(\mathbb{R}^2),\ee
where $\varphi_0$ is a solution to
\be\label{vhr}
-\D\varphi_0=|x|^{-2\ba}e^{8\pi(1-\ba)\varphi_0}\quad {\rm in}\quad \mathbb{R}^2\bs\{0\}.
\ee
Then for any fixed $R>0$, we have
\bna
\int_{\B_{R}(0)}|x|^{-2\ba}e^{8\pi(1-\ba)\varphi_0}dx&\leq &\limsup_{\ve\ra 0}\int_{\B_{R}(0)}|x|^{-2\ba}e^{4\pi(1-\ba-\ve)(1+\psi_\ve)\varphi_\ve}dx\\
&\leq &\limsup_{\ve\ra 0}\l_\ve^{-1}\int_{\B_Rr_\ve^{1/(1-\ba)}(0)}|y|^{-2\ba}c_\ve^2e^{4\pi(1-\ba-\ve)u_\ve^2(y)}dy\\
&\leq &\limsup_{\ve\ra 0}\l_\ve^{-1}\int_{\B_Rr_\ve^{1/(1-\ba)}(0)}|y|^{-2\ba}u_\ve^2(y)e^{4\pi(1-\ba-\ve)u_\ve^2(y)}dy\\
&\leq& 1.
\ena
Therefore,
$$\int_{\mathbb{R}^2}|x|^{-2\ba}e^{8\pi(1-\ba)\varphi_0}dx\leq 1.$$
The classification result of Chen and Li (\cite{CL95}, Theorem 3.1) leads to
\be\label{vor}
\varphi_0(x)=-\f{1}{4\pi(1-\ba)}\log\le(1+\f{\pi}{1-\ba}|x|^{2(1-\ba)}\ri).
\ee
It follows that
\be\label{int}
\int_{\mathbb{R}^2}|x|^{-2\ba}e^{8\pi(1-\ba)\varphi_0}dx= 1.\ee

We now analyze the behavior of $u_\ve$ away from the zero.
Let $u_{\ve,\tau}=\min\{u_\ve, \tau c_\ve\}$ for any $\tau$, $0<\tau<1$. The we have the following lemma.
\begin{lemma}\label{min}
For any $\tau$, $0<\tau<1$, there holds $\lim_{\ve\ra 0}||u_{\ve,\tau}||^2_{\msh,\al}=\tau.$
\end{lemma}
\begin{proof}
By the equation (\ref{Euler-Lagrange}), the integration by parts yields
\bna\ib|\na u_{\ve,\tau}|^2dx&=&\ib\na u_{\ve,\tau}\na u_\ve dx\\
&=&\ib\le(\f{u_{\ve,\tau}u_\ve}{(1-|x|^2)^2}
+\al u_{\ve,\tau}u_\ve+\l_\ve^{-1}|x|^{-2\ba}u_{\ve,\tau}u_\ve e^{4\pi(1-\ba-\ve)u_\ve^2}\ri)dx.
\ena
Thus,
\bna
||u_{\ve,\tau}||^2_{\msh,\al}&=&\ib \le(\na u_{\ve,\tau}\na u_\ve-\f{u_{\ve,\tau}^2}{(1-|x|^2)^2}-\al u_{\ve,\tau}^2\ri)dx\\
&=& \ib\le(   \f{u_{\ve,\tau}(u_\ve-u_{\ve,\tau})}{(1-|x|^2)^2} +\al u_{\ve,\tau}(u_\ve-u_{\ve,\tau})+\l_\ve^{-1}|x|^{-2\ba}u_{\ve,\tau}u_\ve e^{4\pi(1-\ba-\ve)u_\ve^2}\ri)dx\\
&\geq &\ib\le(\l_\ve^{-1}|x|^{-2\ba}u_{\ve,\tau}u_\ve e^{4\pi(1-\ba-\ve)u_\ve^2}\ri)dx\\
&\geq & \tau\int_{\mathbb{B}_{Rr_\ve^{1/(1-\ba)}}}\le(\l_\ve^{-1}|x|^{-2\ba}c_\ve u_\ve e^{4\pi(1-\ba-\ve)u_\ve^2}\ri)dx\\
&\geq& \tau\int_{\mathbb{B}_R(0)}(1+o_\ve(1))|x|^{-2\ba}e^{8\pi(1-\ba)\varphi_0}dx
\ena
for any fixed $R>0$. Letting $\ve\ra 0$ first and then $R\ra \iy$, we get
$$\liminf_{\ve\ra 0}||u_{\ve,\tau}||^2_{\msh,\al}\geq \tau.$$
Noting $|\na(u_\ve-\tau  c_\ve)^+|^2=\na(u_\ve-\tau c_\ve)^+\na u_\ve$ on $\mathbb{B}$ and $(u_\ve-\tau c_\ve)^+=(1+o_\ve(1))(1-\tau)c_\ve$ on $\B_{Rr_\ve^{1/(1-\ba)}}(0)$, we have by the similar calculation that
$$\liminf_{\ve\ra 0}||(u_\ve-\tau c_\ve)^+||^2_{\msh,\al}\geq 1-\tau.$$
Since $|\na u_\ve|^2=|\na u_{\ve,\tau}|^2+|\na (u_\ve-\tau c_\ve)^+|^2$ and $u_\ve \ra 0$ in $L^p(\B)$ for any fixed $p\geq 1$, we have
$$\lim_{\ve\ra 0}\le(||u_{\ve,\tau}||^2_{\msh,\al}+||(u_\ve-\tau c_\ve)^+||^2_{\msh,\al}\ri)=\lim_{
\ve\ra 0}||u_\ve||^2_{\msh,\al}=1.$$
Hence,
$$\lim_{\ve\ra0}||u_{\ve,\tau}||^2_{\msh,\al}=\tau,\quad \lim_{\ve\ra 0}||(u_\ve-\tau c_\ve)^+||^2_{\msh,
\al}=1-\tau.$$
\end{proof}
\begin{lemma}\label{lce}
There holds
$$\limsup_{\ve\ra 0}\ib|x|^{-2\ba}e^{4\pi(1-\ba-\ve)u_\ve^2}\leq \f{\pi}{1-\ba}+\limsup_{\ve\ra 0}\f{\l_\ve}{c^2_\ve}.
$$
\end{lemma}
\begin{proof}
For any $\tau$, $0<\tau<1$, we have
\bna\ib|x|^{-2\ba}e^{4\pi(1-\ba-\ve)u_\ve^2}dx&=&\int_{\{u_\ve\leq \tau c_\ve\}}|x|^{-2\ba}e^{4\pi(1-\ba-\ve)u_\ve^2}dx
+\int_{\{u_\ve>\tau c_\ve\}}|x|^{-2\ba}e^{4\pi(1-\ba-\ve)u_\ve^2}dx\\
&\leq &\ib|x|^{-2\ba}e^{4\pi(1-\ba-\ve)u_{\ve,\tau}^2}dx+\f{\l_\ve}{\tau^2c_\ve^2}.
\ena
By Lemma \ref{min}, we have that $|x|^{-2\ba}e^{4\pi(1-\ba-\ve)u_{\ve,\tau}^2}$ is bounded in $L^p(\B)$ for some
$p\geq 1$. Therefore,
$$\lim_{\ve\ra 0} \ib |x|^{-2\ba}e^{4\pi(1-\ba-\ve)u_{\ve,\tau}^2}dx=\ib |x|^{-2\ba}dx=\f{\pi}{1-\ba}.$$
Combining the above estimates and letting $\ve\ra 0$ first, then $\tau\ra 1$, we finish the proof.
\end{proof}
Similar to  \cite{GY12} and \cite{YZ15}, we prove the following:
\begin{lemma}\label{dlf}
For any $\phi\in C^{\iy}(\overline{\B})$, there holds
$$\lim_{\ve\ra 0}\ib \l_\ve^{-1}|x|^{-2\ba}c_\ve u_\ve e^{4\pi(1-\ba-\ve)u_\ve^2}\phi dx=\phi(0).$$
\end{lemma}
\begin{proof}
Divide $\B$ into three parts:
$$\B=\le(\{u_\ve>\tau c_\ve\}\bs \B _{Rr_\ve^{1/(1-\ba)}}(0)\ri)\cup \le(\{u_\ve\leq \tau c_\ve\}\bs \B _{Rr_\ve^{1/(1-\ba)}}(0)\ri)\cup \B _{Rr_\ve^{1/(1-\ba)}}(0),$$
for  some $0<\tau<1.$
Denote the integrals on the above three domains by $I_1$, $I_2$ and $I_3$ respectively. We estimate them one by one.
In view of (\ref{pgo}), (\ref{vgh}) and (\ref{int}),  we have
\bna
I_1&\leq & \sup_{\B}|\phi|\int_{\{u_\ve>\tau c_\ve\}\bs \B _{Rr_\ve^{1/(1-\ba)}}(0)} \l_\ve^{-1}|x|^{-2\ba}c_\ve u_\ve e^{4\pi(1-\ba-\ve)u_\ve^2} dx\\
&\leq & \f{1}{\tau}\sup_{\B}|\phi|\le(1-\int_{ \B _{Rr_\ve^{1/(1-\ba)}}(0)}\l_\ve^{-1}|x|^{-2\ba} u^2_\ve e^{4\pi(1-\ba-\ve)u_\ve^2} dx\ri)\\
&\leq & \f{1}{\tau}\sup_{\B}|\phi|\le(1-\int_{ \B _{R}(0)}|x|^{-2\ba}  e^{8\pi(1-\ba)\phi_0} dx+o_\ve(R)\ri),
\ena
where $o_\ve(R)\ra 0$ as $\ve \ra 0$ for any fixed $R>0$. Thus, $I_1\ra 0$ by letting  $\ve \ra 0$ first and then $R\ra +\iy.$
Noting that $|x|^{-2\ba}e^{4\pi(1-\ba-\ve)u_{\ve,\tau}^2}$ is bounded in $L^p(\B)$ for some
$p\geq 1$, we have
\bna
I_2&=&\int_{\{u_\ve\leq \tau c_\ve\}\bs \B _{Rr_\ve^{1/(1-\ba)}}(0)}\l_\ve^{-1}|x|^{-2\ba}c_\ve u_\ve e^{4\pi(1-\ba-\ve)u_\ve^2}\phi dx\\
&\leq &\sup_{\B}|\phi|\f{c_\ve}{\l_\ve}\ib|x|^{-2\ba} u_\ve e^{4\pi(1-\ba-\ve)u_{\ve,\tau}^2}\phi dx\\
&\leq &\sup_{\B}|\phi|\f{c_\ve}{\l_\ve}||u_\ve||_{L^q(\B)}|||x|^{-2\ba}  e^{4\pi(1-\ba-\ve)u_{\ve,\tau}^2}||_{L^p(\B)},
\ena
where $1/q+1/p=1$. Lemma \ref{lce} yields $\l_\ve/c_\ve\ra +\iy$, hence $c_\ve/\l_\ve\ra 0$. We obtain $I_2\ra 0$ as $\ve \ra 0$.
Next,
\bna
I_3&=&\int_{\B_{Rr_\ve^{1/(1-\ba)}}(0)}\l_\ve^{-1}|x|^{-2\ba}c_\ve u_\ve e^{4\pi(1-\ba-\ve)u_\ve^2}\phi dx\\
&=&\phi(0)(1+o_\ve(1))\le(\int_{\B_R(0)}|x|^{-2\ba}e^{8\pi(1-\ba)\phi_0}dx+o_\ve(R)\ri).
\ena
Letting $\ve\ra 0$ first and then $R\ra +\iy$, we have $I_3\ra \phi(0)$. Finally, we get
$$\lim_{\ve\ra 0}\left(I_1+I_2+I_3\right)=\phi(0).$$\end{proof}
Define the operator
$$\mathscr{L}_\al=-\D-\f{1}{(1-|x|^2)^2}-\al.$$
We have the following lemma.
\begin{lemma}\label{con} $c_\ve u_\ve\rightharpoonup G$ weakly in $W^{1,p}_{\rm loc}(\B)$ for any $1<p<2$,   strongly in $L^q(\B)$ for any $q\geq 1$
and in $C^0(\overline{\B_r^c})$ for any $0<r<1$, where $G$ is a Green function satisfying
\be\label{gre}
\mathscr{L}_\al(G)=\d_0,
\ee
where $\d_0$ is the Dirac  measure centered at $0$.
\end{lemma}
\begin{proof}
Note that
\be\label{cue}
\mathscr{L}_\al(c_\ve u_\ve)=\f{1}{\l_\ve}|x|^{-2\ba}c_\ve u_\ve e^{4\pi(1-\ba-\ve)u_\ve^2}.\ee
Denote $f_\ve=\f{1}{\l_\ve}|x|^{-2\ba}c_\ve u_\ve e^{4\pi(1-\ba-\ve)u_\ve^2}$. The rest of the proof is the same as in \cite{YZ15}. For completeness, we give the main steps.
Let $\nu_\ve$ be a solution to
\be\le\{\begin{array}{lll} \mss_\al \nu_\ve= f_\ve, \quad {\rm in} \quad \B_{1/2},\\[1.2ex]
\nu_\ve= 0\quad {\rm on} \quad \partial \B_{1/2}.
\end{array}\ri.\ee
Then for any $q$, $1<q<2$, there holds
$$\nu_\ve\rightharpoonup \nu_0\quad{\rm weakly\,\,in}\quad W^{1,q}_0(\B_{1/2}).$$
Set $w_\ve=c_\ve u_\ve-\phi \nu_\ve$, where $\phi$ is a cut-off function in $C^\iy_0(\B)$ with $0\leq \phi\leq 1$, $\phi\equiv 1$ on $\B_{1/8}$ and $\phi\equiv 0$ outside $\B_{1/4}$. Then
\be\label{omgc}
w_\ve\rightharpoonup w_0,\quad{\rm weakly\,\,in}\quad\msh,\ee
and $G=\phi\nu_0+w_0$.

\end{proof}

The Green function $G$ can be represented by
\be\label{gfr}
G=-\f{1}{2\pi}\log r+A_0+\Phi,
\ee
where $A_0$ is a constant and $\Phi\in C^1_{\rm loc}(\B)$.
\subsection{Upper bound estimates}

In this subsection, we will use Iula-Mancini's result to derive the upper bound of the Hardy-Moser-Trudinger functionals.
\begin{lemma}\label{IML16} (Iula-Mancini\cite{IM16})
Let $u_n\in W^{1,2}_0(\B)$ be such that $\ib |\na u_n|^2dx\leq 1$ and $u_n\rightharpoonup 0$ in $W^{1,2}_0(\B)$, then for any fixed $\ba$, $0\leq \ba<1$, we have
\be\label{imli}
\sup_{n\ra\iy}\ib\f{e^{4\pi(1-\ba)u_n^2}}{|x|^{2\ba}}dx\leq \f{\pi(1+e)}{1-\ba}.
\ee
\end{lemma}

We proceed as in \cite{CC86,GY12} and get the following:
\begin{lemma}\label{ubd}
$$\sup_{u\in\msh,\,||u||_{\msh,\alpha}\leq 1}\int_{\mathbb{B}}\f{e^{4\pi(1-\ba)u^{2}}}{|x|^{2\beta}}dx\leq \f{\pi}{1-\ba}\le(1+e^{1+4\pi(1-\ba)A_0}\ri).$$

\end{lemma}
\begin{proof} In view of (\ref{vgh}), we have
$$\int_{\B_{Rr_\ve^{1/(1-\ba)}}}|x|^{-2\ba} e^{4\pi(1-\ba-\ve)u_\ve^2}dx=\f{\l_\ve}{ c_\ve^2} \le(\int_{\B_R}|x|^{-2\ba}e^{8\pi(1-\ba)\varphi_0}dx+o_\ve(R)\ri).$$
Therefore,
$$\lim_{R\ra +\iy}\limsup_{\ve\ra 0}\int_{\B_{Rr_\ve^{1/(1-\ba)}}}|x|^{-2\ba} e^{4\pi(1-\ba-\ve)u_\ve^2}dx=\limsup_{\ve\ra 0}\f{\l_\ve}{c^2_\ve}.$$
Let $\rho \in (0,1)$. By Lemma \ref{con}, we have
\be
\lim_{\ve\ra 0}c_\ve u_\ve(\rho)=G(\rho)\ee
and
\be
\lim_{\ve\ra 0}\int_{\B_\rho}\le(\f{1}{(1-|x|^2)^2}+\al\ri)(c_\ve u_\ve)^2dx=\int_{\B_\rho}\le(\f{1}{(1-|x|^2)^2}+\al\ri)G^2dx=:E_1(\rho).
\ee
By (\ref{cue}), we have
\bna&&\int_{\B_\rho^c}\le(|\na(c_\ve u_\ve)|^2-\f{c_\ve^2u_\ve^2}{(1-|x|^2)^2}-\al c_\ve^2u^2\ri)dx\\
&&=||c_\ve u_\ve||^2_{\msh,\al}
-\int_{\B_\rho}\le(|\na(c_\ve u_\ve)|^2-\f{c_\ve^2u_\ve^2}{(1-|x|^2)^2}-\al c_\ve^2u^2\ri)dx\\
&&=\int_{\B_\rho^c}\l_\ve^{-1}|x|^{-2\ba}(c_\ve u_\ve)^2e^{4\pi(1-\ba-\ve)u_\ve^2}dx-\int_{\partial \B_\rho}\f{\partial(c_\ve u_\ve)}{\partial \nu}(c_\ve u_\ve)d\sigma.
\ena
Clearly,
$$\int_{\B_\rho^c}\l_\ve^{-1}|x|^{-2\ba}(c_\ve u_\ve)^2e^{4\pi(1-\ba-\ve)u_\ve^2}dx\ra 0, \quad {\rm as}\quad \ve\ra 0,$$
and
\bna-\int_{\partial \B_\rho}\f{\partial(c_\ve u_\ve)}{\partial \nu}(c_\ve u_\ve)d\sigma &=&-c_\ve u_\ve(\rho)\int_{\B_\rho}\D(c_\ve u_\ve)dx\\
&=& c_\ve u_\ve(\rho)\le(\int_{\B_\rho}\f{c_\ve u_\ve}{(1-|x|^2)^2}dx +\int_{\B_\rho}\al c_\ve u_\ve dx\ri)\\
&&+c_\ve u_\ve(\rho)\int_{\B_\rho}\l_\ve^{-1}|x|^{-2\ba}c_\ve u_\ve e^{4\pi(1-\ba-\ve)u_\ve^2}dx\\
&\ra & G(\rho)\le(\int_{\B_\rho}\f{G}{(1-|x|^2)^2}dx+\int_{\B_\rho}\al Gdx+1\ri)=:E_2(\rho).
\ena
Hence
\bea
\int_{\B_\rho}|\na u_\ve|^2dx &=&1-\int_{\B_\rho^c}\le(|\na u_\ve|^2-\f{u^2}{(1-|x|^2)^2}-\al u^2\ri)dx \nonumber\\
&&+\int_{\B_\rho}\le( \f{u_\ve^2}{(1-|x|^2)^2}+\al u_\ve^2\ri)dx\nonumber\\
&=&1-\f{1}{c_\ve^2}\le[E_2(\rho)-E_1(\rho)+o_\ve(1)\ri].
\eea
Let $F_\rho :=E_2(\rho)- E_1(\rho)$. Then
$$\int_{\B_\rho}|\na u_\ve|^2dx=1-\f{F_\rho+o_\ve(1)}{c_\ve^2}.$$
Let $\overline{u}=[u_\ve-u_\ve(\rho)]^+$ and $s_\ve=\overline{u}/||\na \overline{u}||_{L^2(\B_\rho)}$. Obviously, $s_\ve\in W_0^{1,2}(\B_\rho)$,
$||\na s_\ve||_2=1$ and $s_\ve\rightharpoonup 0$ in $W_0^{1,2}(\B_\rho)$.
By (\ref{pgo}), we have that $c_\ve^{-1}s_\ve\ra 1$ uniformly on $\B_{Rr_\ve^{1/(1-\ba)}}$.
Therefore, we have
\bna
u_\ve^2(x)&\leq &\le[ s_\ve(x)+u_\ve(\rho)\ri]^2||\na u_\ve||^2_{L^2(\B_\rho)}\\
&=& \le[s_\ve(x)+c_\ve^{-1}G(\rho)+o_\ve(c_\ve^{-1})\ri]^2\times\le[1-c_\ve^{-2}F_\rho +o_\ve(c_\ve^{-2})\ri]\\
&=& s_\ve^2(x)+2G(\rho)-F_\rho +o_\ve(1),
\ena
where $o_\ve(1)$ goes to $0$ uniformly in $\B_{Rr_\ve^{1/(1-\ba)}}$. According to  (\ref{imli}),  we get
\bna
\limsup_{\ve\ra 0}\int_{\B_{Rr_\ve^{1/(1-\ba)}}}|x|^{-2\ba}e^{4\pi(1-\ba-\ve)u_\ve^2}dx&\leq &\limsup_{\ve\ra 0}\int_{\B_{Rr_\ve^{1/(1-\ba)}}}|x|^{-2\ba}\le(e^{4\pi(1-\ba)u_\ve^2}-1\ri)dx\\
&\leq & e^{4\pi(1-\ba)(2G(\rho)-F_\rho)}\limsup_{\ve\ra 0}\int_{\B_{Rr_\ve^{1/(1-\ba)}}}|x|^{-2\ba}\le(e^{4\pi (1-\ba)s_\ve^2}-1\ri)dx\\
&\leq & e^{4\pi(1-\ba)(2G(\rho)-F_\rho)}\limsup_{\ve\ra 0}\int_{\B_\rho}|x|^{-2\ba}\le(e^{4\pi (1-\ba)s_\ve^2}-1\ri)dx\\
&\leq & \pi(1-\ba)^{-1} \rho^{2(1-\ba)}e^{1+4\pi(1-\ba)(2G(\rho)-F_\rho)}.
\ena
In view of (\ref{gfr}), we obtain
\bna\pi(1-\ba)^{-1} \rho^{2(1-\ba)}e^{1+4\pi(1-\ba)(2G(\rho)-F_\rho)}&=&\pi(1-\ba)^{-1}e^{1+4\pi(1-\ba)\le(\f{1}{2\pi}\log \rho+2G(\rho)-F_\rho\ri)}\\
&\ra&\pi(1-\ba)^{-1}e^{1+4\pi(1-\ba)A_0},
\ena as $\rho\ra 0$.
Combining Lemma \ref{lce}, we finish the proof.\end{proof}
\subsection{The existence result}
If $c_\ve$ is bounded, then our theorem holds true. If $c_\ve\ra +\iy$, we will construct a sequence of functions $\phi_\ve(x)\in \msh$ satisfying $||\phi_\ve(x)||_{\msh,\al}\leq 1$ and
 $$\ib |x|^{-2\ba}e^{4\pi(1-\ba)\phi_\ve^2}dx>\pi (1-\ba)^{-1}(1+e^{1+4\pi(1-\ba)A_0}).$$
 This is a contradiction. Then the proof of Theorem \ref{mai1} is completed since $c_\ve$ must be bounded.

Let

\be\label{tsf}\phi_\ve(x)=\le\{\begin{array}{lll}c+\f{1}{c}\le(-\f{1}{4\pi(1-\ba)}\log(1+\f{\pi}{1-\ba}\f{|x|^{2(1-\ba)}}{\ve^{2(1-\ba)}})+b\ri)
& x\in\overline{\B}_{R\ve},&\\[2ex]
\f{G(x)}{c} &x\in \B\bs\B_{R\ve},&
\end{array}\ri.\ee
where $R=(-\log \ve)^{1/(1-\ba)} $, $b$ and $c$ are constants to be determined later. We require
\be
c+\f{1}{c}\le(-\f{1}{4\pi(1-\ba)}\log(1+\f{\pi}{1-\ba}R^{2(1-\ba)})+b\ri)=\f{1}{c}\le(-\f{1}{2\pi}\log(R\ve)+A_0+O(R\ve)\ri),
\ee
which gives
\be\label{cr1}
c^2=-\f{1}{2\pi}\log \ve +A_0-b+\f{1}{4\pi(1-\ba)}\log\f{\pi}{1-\ba}+O(\f{1}{R^{2(1-\ba)}}).
\ee
By (\ref{omgc}),  $G=w_0$ on $\B\bs \B_{1/2}$. Noting $w_0\in \msh$ and $\phi_\ve-w_0/c\in W^{1,2}_0(\B)$, we get $\phi_\ve\in\msh$.
We have by integration by parts that
\bna
\int_{\B\bs \B_{R\ve}}\le( |\na \phi_\ve|^2-\f{\phi_\ve^2}{(1-|x|^2)^2}-\al\phi_\ve^2\ri)dx&=&\f{1}{c^2}\int_{\partial \B_{R\ve}}G\f{\partial G}{\partial \nu}d\sigma +\f{1}{c^2}\int_{\B\bs\B_{R\ve}}G\mathscr{L}_\al Gdx\\
&=&\f{1}{c^2}\le(-\f{1}{2\pi}\log(R\ve)+A_0+O(R\ve)\ri).\ena
A direct calculation   shows that
\be\label{cal1}
\int_{\B_{R\ve}}|\na \phi _\ve|^2dx=\f{1}{4\pi(1-\ba)c^2}\le( \log \f{\pi}{1-\ba}+\log R^{2-2\ba}-1+O(\f{1}{R^{2-2\ba}})\ri).
\ee
Refer to \cite{YZ16} for detailed calculations for (\ref{cal1}) and the following (\ref{cal2}).
Then we have
\bna
||\phi_\ve||^2_{\msh,\al}&\leq & \int_{\B\bs \B_{R\ve}}\le( |\na \phi_\ve|^2-\f{\phi_\ve^2}{(1-|x|^2)^2}-\al\phi_\ve^2\ri)dx+\int_{\B_{R\ve}}|\na \phi_\ve|^2dx\\
&=& \f{1}{c^2}\le(-\f{1}{2\pi}\log \ve+A_0+\f{1}{4\pi(1-\ba)}\log\f{\pi}{1-\ba}-\f{1}{4\pi(1-\ba)}+O(\f{1}{R^{2-2\ba}})\ri).
\ena
Letting the last term in above inequality equals $1$, we get
\be\label{cr2}
c^2=-\f{1}{2\pi}\log \ve+A_0+\f{1}{4\pi(1-\ba)}\log\f{\pi}{1-\ba}-\f{1}{4\pi(1-\ba)}+O(\f{1}{R^{2-2\ba}}).
\ee
Combining (\ref{cr1}) and (\ref{cr2}), we get
\be
b=\f{1}{4\pi(1-\ba)}+O(\f{1}{R^{2-2\ba}}).
\ee
We also have
\be\label{cal2}
\int_{\B_{R\ve}}|x|^{-2\ba}e^{4\pi(1-\ba)\phi_\ve^2}dx\geq  \f{\pi}{1-\ba}e^{1+4\pi(1-\ba)A_0}+O(\f{1}{R^{2-2\ba}}).
\ee
Next,  we calculate
\bna
\int_{\B\bs\B_{R\ve}}|x|^{-2\ba}e^{4\pi(1-\ba)\phi_\ve^2}dx&\geq &\int_{\B\bs\B_{R\ve}}|x|^{-2\ba}(1+4\pi(1-\ba)\phi_\ve^2)dx\\
&=&\int_{\B\bs\B_{R\ve}}|x|^{-2\ba}dx+\f{4\pi(1-\ba)}{c^2}\int_{\B\bs\B_{R\ve}}|x|^{-2\ba}G^2dx\\
&=&\int_{\B}|x|^{-2\ba}dx-\int_{\B_{R\ve}}|x|^{-2\ba}dx+\f{4\pi(1-\ba)}{c^2}\int_{\B}|x|^{-2\ba}G^2dx\\
&&-\f{4\pi(1-\ba)}{c^2}\int_{\B_{R\ve}}|x|^{-2\ba}G^2dx\\
&=&\f{\pi}{1-\ba}+\f{4\pi(1-\ba)}{c^2}\int_{\B}|x|^{-2\ba}G^2dx+O(\f{1}{R^{2-2\ba}}).
\ena
Noting $O(\f{1}{R^{2-2\ba}})=o(\f{1}{c^2})$  and combining the above estimates, we get
$$\ib |x|^{-2\ba}e^{4\pi(1-\ba)\phi_\ve^2}dx>\pi(1-\ba)^{-1}\le(1+e^{1+4\pi(1-\ba)A_0}\ri),$$
if $\ve$ is sufficiently small.


\begin{thebibliography}{00}


\bibitem{AD04} Adimurthi, O. Druet, Blow-up analysis in dimension 2 and a sharp form of Trudinger-Moser inequality, Comm.
Partial Differential Equations 29 (2004) 295-322.
\bibitem{A-S} Adimurthi, K. Sandeep, A singular Moser-Trudinger
embedding and its applications, Nonlinear Differ. Equ. Appl. 13
(2007) 585-603.
\bibitem{AST00} Adimurthi, M. Struwe, Global compactness properties of semilinear elliptic equation with critical exponential
growth, J. Functional Analysis 175 (2000) 125-167.

\bibitem{AY10} Adimurthi, Y. Yang, An interpolation of Hardy inequality
and Trudinger-Moser inequality in $\mathbb{R}^N$ and its
applications, Int. Math. Res. Notices 13 (2010)
2394-2426.
\bibitem{Ba94} A. Baernstein, A unified approach to symmetrization, Partial differential equations of elliptic type (Cortona, 1992),
Sympos. Math 35 (1994) 47-91.
\bibitem{HM97} H. Brezis, M. Marcus, Hardy¡¯s inequality revisited, Ann. Scuola Norm. Pisa 25 (1997) 217-237.
\bibitem{CC86} L. Carleson, A. Chang, On the existence of an extremal function for an inequality of J. Moser, Bull. Sci. Math. 110
(1986) 113-127.
\bibitem{CL91} W. Chen, C. Li, Classification of solutions of some nonlinear elliptic equations, Duke Math. J. 63 (1991) 615-622.
\bibitem{CL95} W. Chen, C. Li, What kinds of singular surfaces can admit constant curvature? Duke Math. J. 78 (1995) 437-451.
\bibitem{CSP15} Gyula Csat\'{o},  Prosenjit Roy, Extremal functions for the singular Moser-Trudinger
inequality in 2 dimensions, Calc. Var. 54 (2015) 2341-2366.
\bibitem{DJL97} W. Ding, J. Jost, J. Li, G.Wang, The differential equation $-\D u=8\pi-8\pi he^u$ on a compact Riemann Surface, Asian
J. Math. 1 (1997) 230-248.
\bibitem{flu92} M. Flucher, Extremal functions for Trudinger-Moser inequality in 2 dimensions, Comment. Math. Helv. 67 (1992)
471-497.
\bibitem{IM16} Stefano Iula, Gabriele Mancini, Extremal functions for singular Moser-Trudinger embeddings, arXiv: 1601.05666v1.
\bibitem{LY16} X. Li, Y. Yang, Extremal functions for singular Trudinger-Moser inequalities in the entire Euclidean space, arXiv:1612.08247.
\bibitem{Li01} Y.  Li, Moser-Trudinger inequality on compact Riemannian manifolds of dimension two, J. Part. Diff. Equations 14
(2001) 163-192.
\bibitem{Lin96} K. Lin, Extremal functions for Moser¡¯s inequality, Trans. Amer. Math. Soc. 348 (1996) 2663-2671.
 \bibitem{MS08}  G. Mancini, K. Sandeep, On a semilinear elliptic equation in $\mathbb{H}^n$, Ann. Sc. Norm. Super. Pisa Cl. Sci. 7 (2008)   635-671.
\bibitem{Mos71}J. Moser, A sharp form of an inequality by N.Trudinger, Ind. Univ. Math. J. 20 (1971) 1077-1091.
\bibitem{str88} M. Struwe, Critical points of embeddings of $H_0^{1,n}$
 into Orlicz spaces, Ann. Inst. H. Poincar\'{e}, Analyse Non Lin\'{e}aire
5 (1988) 425-464.

\bibitem{Tin14} C. Tintarev, Trudinger-Moser inequality with remainder terms, J. Funct. Anal. 266 (2014) 55-66.
\bibitem{Tru67} N. Trudinger, On embeddings into Orlicz spaces and some applications, J. Math. Mech. 17 (1967) 473-484.
\bibitem{GY12} G. Wang, D. Ye, A Hardy-Moser-Trudinger inequality, Adv. Math. 230 (2012) 294-320.
\bibitem{Yang06} Y. Yang, A sharp form of Moser-Trudinger inequality in high dimension, J. Funct. Anal. 239 (2006) 100-126.
\bibitem{Yang07}  Y. Yang,  A sharp form of trace Moser-Trudinger inequality on compact Riemannian surface with boundary,
Trans. Amer. Math. Soc. 359 (2007) 5761-5776.

\bibitem{Yang15} Y. Yang, Extremal functions for Trudinger-Moser inequalities of Adimurthi-Druet type in dimension two, J. Differential Equations 258 (2015) 3161-3193.
\bibitem{YZ15} Y. Yang, X. Zhu,  An improved Hardy-Trudinger-Moser inequality, Ann. Glob. Anal. Geom. 49 (2016) 1-19.
\bibitem{YZ16} Y. Yang, X. Zhu,  Blow-up analysis concerning singular Trudinger-Moser inequalities in dimension two,  J. Funct. Anal. 272 (2017) 3347-3374.
\bibitem{YuZ}  A. Yuan, X. Zhu, An improved singular Trudinger-Moser inequality in unit ball, J. Math. Anal. Appl. 435 (2016)
244-252.




\end{thebibliography}
\end{document}